\documentclass[reqno]{amsart}
\usepackage{color}
\usepackage{mathrsfs}
\usepackage{hyperref}
\usepackage{enumerate}
\usepackage{amsmath}
\usepackage{amsthm}
\usepackage{amssymb}
\usepackage{amscd}
\usepackage{graphicx}
\usepackage{epsfig}
\usepackage[english]{babel}
   \newcommand{\g}{\gamma}

  \newcommand{\w}{\omega}
  
\newcommand{\R}{\mathbb{R}}
\newcommand{\C}{\mathbb{C}}

\newcommand{\GG}{\mathbb{G}}
\newcommand{\EE}{\mathbb{E}}

\newcommand{\OO}{\mathbb{O}}

\newcommand{\PP}{\mathbb{P}}

\newcommand{\U}{\mathbb{U}}

\newcommand{\Z}{\mathbb{Z}}

   \newcommand{\cD}{\mathcal{D}}
   \newcommand{\cF}{\mathcal{F}}
   \newcommand{\cG}{\mathcal{G}}
   
        \newcommand{\cK}{\mathcal{K}}
    \newcommand{\cH}{\mathcal{H}}
     \newcommand{\cI}{\mathcal{I}}
    \newcommand{\cL}{\mathcal{L}}

  \newcommand{\cN}{\mathcal{N}}

  \newcommand{\cO}{\mathcal{O}}

\renewcommand{\g}{\mathfrak{g}}  % Lie algebra

\newcommand{\h}{\mathfrak{h}}          % Lie subalgebra
\newcommand{\tto}{\rightrightarrows}    % Arrows of a groupoid

\DeclareMathOperator{\Ad}{Ad}           % Adjoint
\allowdisplaybreaks

\newtheorem{theorem}{Theorem}
\newtheorem{lemma}[theorem]{Lemma}
\newtheorem{proposition}[theorem]{Proposition}
\newtheorem{corollary}[theorem]{Corollary}

\theoremstyle{definition}
\newtheorem{definition}[theorem]{Definition}
\newtheorem{example}[theorem]{Example}

\newtheorem{remark}[theorem]{Remark}

\newtheorem{theo}{Theorem}

\begin{document}
%%%%%%%%%%%%%%%%%%%%%%%%%%%%%%%%%%%%%%%%%%
% \title[PMCT 2]{Hilbert V}

% 
\title{Proper Lie groupoids are real analytic}

 \author{David Mart\'inez Torres}
\address{Departamento de Matem\'atica, PUC-Rio, R. Mq. S. Vicente 225, Rio de Janeiro 22451-900, Brazil}
\email{dfmtorres@gmail.com}

\begin{abstract}
We show that any proper Lie groupoid admits a compatible real analytic structure. 
Our proof hinges on a Weyl unitary trick of sorts 
for appropriate local holomorphic groupoids.

\end{abstract}

\maketitle

\section{Introduction and statement of results}

% 
% In connection with Hilbert's fifth problem/Hilbert-Smith conjecture, Illman proved that proper group actions 
% can be made $C^\w$ (real analytic) \cite{Il,Il2}. We prove the same statement for
%  proper Lie groupoids:
% 
% 
% \begingroup
% \setcounter{tmp}{\value{theo}}% store current value of theorem counter
% \setcounter{theo}{0} %assign desired value to theorem counter
% \renewcommand\thetheo{\Alph{theo}}% locally redefine the representation of the theorem counter
% \begin{theo}\label{thm:main}
%  Any proper Lie groupoid $\cG\tto M$  admits a  compatible $C^\w$ structure.
% \end{theo}
% \endgroup
% 
% 
% The technical result behind the proof of Theorem \ref{thm:main} is to show that Zung's iteration method \cite{Zu} is compatible
% with local $C^\w$ structures.  This is done by extending Zung's averaging step 
% to appropriate local holomorphic groupoids. Our technique can be understood as a Weyl unitary trick of sorts for such
% local holomorphic groupoids. 
% Its main ingredient is the introduction of \emph{proper cores}. These are appropriate replacements of the compact form 
% of a complex reductive Lie group whose construction requires basic Kempf-Ness theory for principal bundles.
% 
% 
%  
% 
% The author is grateful to P. Frejlich, E. Miranda and G. Trentinaglia for useful discussions on the subject.
% 
% \subsection{Regularity of group actions}

This paper is concerned with the improvement of regularity of partially defined actions,
a question that can be traced back to Hilbert.

The most common formulation of Hilbert's fifth problem is whether a locally euclidean 
topological group is a Lie group, i.e., whether it possesses a $C^\w$ structure such 
that the group operations become $C^\w$. 
This was answered in the affirmative in the fifties by Gleason, Montgomery and Zippin \cite{Gl,MZ}. The hardest part of their work is to pass from
mere continuity to enough regularity to have a Lie algebra; real analyticity comes as a consequence of the 
Baker-Campbell-Hausdorff formula. 

In Hilbert's time groups were local transformation groups, and therefore 
other formulations of Hilbert's fifth problem exist.
The local Hilbert fifth problem asks whether a locally euclidean topological group is a local Lie group, and it was recently 
answered in the affirmative by Goldbring \cite{Go}. The Hilbert-Smith conjecture asks whether a locally compact
topological group acting effectively on a topological manifold is a Lie group; it has been
recently answered in the affirmative
by Pardon for manifolds of dimension three \cite{Par}. (See \cite{Ta}
for more details on Hilbert's fifth problem and related topics). 

In connection with Hilbert's question, it is natural to ask for sufficient conditions
for a Lie group action
on a manifold to be made $C^\w$. Firstly, a minimal amount of regularity is needed, as
examples do exist from the fifties by Bing and Montgomery and Zippin of
continuous actions of Lie groups which cannot be made $C^1$ \cite{B,MZ2}. Secondly, real analyticity cannot be achieved in general,
as for smooth actions the foliation defined on the regular set of the action cannot have
flat holonomy. The holonomy obstruction is not present for proper actions due to the existence of 
semilocal normal forms coming from Palais' slice Theorem. It was proved in the nineties by Illman that proper $C^1$-actions
can be made $C^\w$ \cite{Il,Il2}.

A Lie group action on a manifold is an example of Lie groupoid. Lie groupoids formalise the notion
of manifolds endowed with partially defined symmetries. 
Besides Lie group actions, other important
examples of Lie groupoids are the fundamental groupoid of a manifold, the holonomy groupoid of a foliation, and
the Lie groupoids associated to (integrable) Poisson and related structures (see \cite{MM,CF} for general theory 
on Lie groupoids
and their connection to foliation theory and Poisson geometry). 
Any Lie groupoid $\cG\tto M$ defines
a possibly non regular foliation on its manifold of base points $M$. This so-called characteristic foliation
is well-behaved for the class of proper Lie groupoids. These are Lie groupoids which generalise 
proper Lie group actions and for 
which  an analogue of Palais' slice Theorem exists \cite{CS,Zu}. Thus, it is natural to ask if a (smooth) proper Lie groupoid
can be given a compatible $C^\w$ structure. Our main result answers this question in the affirmative:

\begin{theo}\label{thm:main}
 Any proper Lie groupoid $\cG\tto M$  admits a  compatible $C^\w$ structure.
\end{theo}

Theorem \ref{thm:main} has the following immediate consequence:

\begin{corollary}\label{cor:charfol} For any proper Lie groupoid $\cG\tto M$  there exists a compatible
 $C^\w$-structure on $M$ such that the leaves of the characteristic foliation are (embedded) real analytic submanifolds 
 (of possibly varying dimension).
\end{corollary}

Corollary \ref{cor:charfol} can be specialised for two distinguished families of proper Lie groupoids:

A (smooth) foliation $\cF$ on a manifold $M$ is called proper if its holonomy groupoid is proper (see \cite[Section 2]{CFM2} for a discussion on foliation groupoids).
Proper foliations include the so-called compact foliations with finite holonomy \cite{Ep,EMS}.

\begin{corollary}\label{cor:fol} A proper foliation $\cF$ on $M$ can be made $C^\w$, i.e., there exists
 a compatible real analytic structure on $M$ for which the leaves of $\cF$ a real analytic (embedded) submanifolds.
\end{corollary}
 
A Poisson manifold $(M,\pi)$ is called of proper type if there exists a proper symplectic Lie groupoid 
which integrates $(M,\pi)$ (see \cite{CFM1} for background on these class of Poisson manifolds).

\begin{corollary} If $(M,\pi)$ is a Poisson manifold of proper type (or, more generally, a Dirac manifold of proper type),
 then there exists a compatible $C^\w$-structure on $M$ for which the symplectic leaves of $(M,\pi)$ are real analytic submanifolds.
\end{corollary}

A feature common to the solutions of all variations of Hilbert's fifth problem is the application of averaging techniques
by means of (perhaps local) Haar measures
to gain access to the geometry at small scale. Our proof of Theorem \ref{thm:main} also hinges on an appropriate averaging 
scheme which can be thought of as a generalisation of Weyl's unitary trick to the groupoid setting.
We briefly describe how it lies at the centre of our proof: 

The fundamental result we need is a refinement of the semilocal normal forms around a leaf of a proper Lie  groupoid.
The refinement amounts to making the construction compatible with a $C^\w$ structure which
is defined on an open saturated subset whose boundary 
contains the leaf around which we work. The difficulty is to recognise how   
the existing $C^\w$-structure reflects on the geometry at small scale around the leaf;
infinitesimal data centred at the leaf does not see this real analytic structure.

The key step to produce semilocal normal forms for proper Lie groupoids   
is an analogue of Bochner's linearization
of a compact group action around a fixed point, and it was developed by Zung \cite[Sections 2.2 to 2.5]{Zu} (there are other approaches but they use infinitesimal
data at the fixed point \cite{CS,HF}): the compact group action is replaced
by an s-proper Lie groupoid with a fixed point. The objective
is to produce a morphism from the s-proper Lie groupoid to the isotropy Lie group
of its fixed point. The method involves averaging with a (normalised) Haar density 
of the Lie groupoid, which is a replacement of
the Haar density of a compact group. The fundamental difference is that for s-proper Lie groupoids
averaging once does not produce the sought after Lie groupoid morphism. In analogy with the centre of
mass construction \cite{GHS},
averaging has to be iterated and convergence of the corresponding sequence of maps has to be proved.

The averaging step in Zung's linearization Theorem is compatible with $C^\w$ structures defined
on arbitrary open saturated subsets. The core of this paper develops tools to show that the limit of 
the corresponding sequence of real analytic 
maps is also real analytic (see Corollary \ref{cor:zunganal}). This result was conjectured by Zung \cite[Remark after Theorem
2.3]{Zu}). 

There exist two characterisations of real analytic functions/maps: A quantitative one
given by appropriate uniform bounds on all derivatives on compact subsets (see e.g. \cite[Proposition 1.1.14]{Na}). A qualitative one presenting
the function as the real part of a holomorphic map equivariant with respect to complex conjugation. 
We found the quantitative characterisation difficult to apply in our case. The reason is that the quantitative analysis of 
Zung's averaging step would need a very fine quantitative analysis 
of some of the `explicit' expansions of the BCH
formula (Dynkin, Magnus or Fer expansions). 

Our qualitative approach amounts to `complexifying' the whole averaging step. We 
propose an averaging scheme for appropriate local holomorphic groupoids that 
is a Weyl unitary trick of sorts. The role of the compact real form is now played by  what we call
an s-proper core which supports a holomorphic Haar density. This is a pair designed to average  holomorphic functions. 
The existence of this semigroup-like subset and Haar density with strong compatibility properties with the holomorphic structure
is very delicate. Their construction in our setting requires basic Kempf-Ness theory of principal bundles. The outcome
is that the sequence of real analytic maps complexifies to a sequence of holomorphic maps with \emph{common} support 
which converges uniformly to a morphism of local holomorphic groupoids.

The paper is written assuming that the reader has some background of (proper)
Lie groupoids, and it is structured as follows: in Section \ref{sec:core} we propose an approach to perform
averaging on local (analytic) Lie groupoids based on s-proper cores endowed with Haar systems. Section
\ref{sec:models} contains the construction s-proper cores and holomorphic Haar systems for the
local complexifications of the s-proper groupoids which appear in the semilocal normal form theorems.  Section
\ref{sec:average} proves Theorem \ref{thm:average} which is the main technical result of this paper: It is an extension
of Zung's iteration method to correct almost morphism to the setting of local analytic groupoids endowed with s-proper cores and analytic Haar systems. 
Section \ref{sec:proof} contains the proof of Theorem \ref{thm:main}. 

 The author is grateful to P. Frejlich, E. Miranda and G. Trentinaglia for useful discussions on the subject.

% (possibly local) Haar measures 
% are used 
% \begin{itemize}
%  \item in the classical and local Hilbert's fifth problem to produce Gleason-Yamabe metrics;
%  \item in the 3-dimensional Hilbert-Smith conjecture to construct 
%  subsets of a 3-manifold endowed with an efective action of the $p$-adic integers which are invariant
% and have suitable homological properties;
% \item in the proof of real analyticity of proper group actions to construct 
% slices through a given point
% which are $C^\w$ in their intersection with any given open subset (not necessarily containing the point) 
% where the action is already $C^\w$.
% \end{itemize}

\section{Averaging on local Lie groupoids}\label{sec:core}
In this section we propose an approach to define (partial) averaging on appropriate local Lie groupoids.

\subsection{Local lie groupoids}
By a Lie groupoid  we mean a $C^\infty$ groupoid. 
By a local Lie groupoid we understand a pair of manifolds and surjective submersions $\cH\tto \cH_0$,
where the axioms of a Lie groupoid are relaxed in the following way:
\begin{itemize}
 \item The inversion is defined on an open neighbourhood of the units.
 \item The multiplication map $m$ is defined on an open neighbourhood of  
 $\cH^{(2)}=\{(q,p)\in \cH^2\,|\, s(q)=t(p)\}$ with contains
 the diagonal embedding of the units. 
 \item For any $p,q,r\in \cH$ such that $(r,q),(q,p)$ can be multiplied and $(m(r,q),p)$ can be multiplied,
 then $(r,m(q,p))$ can be multiplied and $m(m(r,q),p)=(r,m(q,p))$.
\end{itemize}

Given any open subset of $\cH_0$ or more generally an open subset of $\cH$ intersecting the units, there is
a standard way of restricting the local Lie groupoid to the open subset. This leads to a natural notion of
germ of local Lie groupoids around the units or around other bigger subsets. Certain representatives of a germ
of a local Lie groupoid may have additional advantages. For example, it is customary to assume in the definition of a local Lie groupoid
that the inversion is everywhere defined, but this is not really needed at the germ level. In our definition
of local Lie groupoid we have avoided symmetry with respect to the inversion. The reason is
that for our purposes of averaging we are going to multiply from the left.

Any Lie groupoid is an example of a local Lie groupoid. The local Lie groupoids which enter in the proof
of \ref{thm:main} fall in the following class: 

\begin{example} (Complexifications)
 Let $\cG\tto \cG_0$ be a real analytic groupoid. Its complexification is a well defined germ of holomorphic groupoid
 around $\cG\tto \cG_0$ endowed with an antiholomorphic involution. 
 
 A representative is given by first fixing a
 complexification of base points and arrows: $\cH_0,\cH$.
 Then source, target, embedding into units and inversion map are complexified and commute with the antiholomorphic
 involutions
 on units and arrows. Without loss of generality,
 the base points $\cH_0$ can be chosen to be an open neighbourhood of $\cG_0$ (invariant under the antiholomorphic involution)
 so that both source and target maps are (surjective) submersions. The complex manifold  $\cH^{(2)}$
 with the restriction of the product involution on $\cH^2$ is a complexification 
 of $\cG^{(2)}$. Therefore the open subset where the multiplication map is defined can be assumed
 to contain $\cG^{(2)}$ and the embedding of the units $\cH_0$ (and also to be 
 invariant by the product antiholomorphic 
 involution). 
 
 The structural maps are given by equations on (perhaps open subsets of) $\cH_0,\cH,\cH^{(2)}$ and $\cH^{(3)}$.
 These are complexifications of the respective real analytic manifolds for $\cG\tto \cG_0$. By analytic continuation if
 $\cH_0$ is chosen as above and it is a small enough neighbourhood of $\cG_0$, then $\cH\tto \cH_0$ is a local
 holomorphic groupoid.

\end{example}

\subsection{Cores for local Lie groupoids}
Roughly speaking, a core for a local Lie groupoid is a submanifold whose arrows can be used to multiply
from the left with no restriction, and which has some amount of right-homogeneity.

\begin{definition}\label{def:cores}
 Let $\cH\tto \cH_0$ be a local Lie groupoid.
 A core for  $\cH$ is a subset  $\cK \subset  \cH$ with the following properties:
 \begin{enumerate}
  \item(Lie type) $\cK$ is a submanifold of $\cH$ and the restriction of the source map $s_{\cK}:\cK\to \cH_0$ 
  is a surjective submersion (if $\cH\tto \cH_0$  is either real or complex analytic, then $\cK$ is required to be a real analytic submanifold).
%   and the restriction of the target map has constant
%   rank. In particular  $N:=t_{\cK}(\cK)$ is a submanifold $N\subset M$ over which we have the Lie semigroupoid $\cK_N\subset \cK$.
  \item($s_{\cK}$-fiber invertibility) The right action of $\cK$ on $\cK$ (where defined) identifies $s_{\cK}$-fibres.
  \item (no escape) The multiplication map is defined in the submanifold
  \[\cH^{(2)}_\cK:=\{(k,p)\in \cH^{(2)},\,k\in \cK\}.\]
   \end{enumerate}
   We say that $\cK$ is s-proper whenever $s_{\cK}$ is a proper map.

\end{definition}

\begin{example}\label{ex:weyl1} (Weyl's unitary trick I)
Let $G$ be a compact group acting on a vector space $E$ by orthogonal transformations.
 Let $\GG\ltimes \EE$ denote the complexified action, 
let $\EE^r$ be the Euclidean ball of radius $r$ and let $\cH^r\subset \GG\ltimes \EE$ be the open set of arrows
with source and target in $\EE^r$. The s-proper subgroupoid $\cK=G\ltimes \EE^r\subset \cH^r$ is an s-proper
core for $\cH^r$.
\end{example}

\subsection{Haar densities on s-proper cores}
Normalised Haar densities on s-proper Lie groupoids are the analogues of Haar densities on compact Lie groups (see \cite{CM} for 
a detailed treatment of measures on Lie groupoids).
We introduce their natural generalisation to s-proper cores for local analytic groupoids:
\begin{definition} Let $\cH\tto \cH_0$ be a local analytic Lie groupoid (real or complex)
and let $\cK\subset \cH$ be an  s-proper core for $\cH$.
 A (normalised) analytic Haar density on $\cK$ is a real analytic section $\mu$ of the bundle of $\R/\C$-valued  $s_\cK$-fibred densities of $\cK$
 with the following properties:
 \begin{enumerate}
 \item \[\int_{\cK_z}\mu^z=1,\qquad \forall z\in \cH_0.\]
 \item The right action of $\cK$ on $\cK$ (where defined) preserves $\mu$.
 \item Integration along the $s_\cK$-fibres preserves the regularity of functions
 \[\int_\cK\mu:C^\w(\cH)/\mathrm{Hol}(\cH)\to C^\w(\cH_0)/\mathrm{Hol}(\cH_0),\]
 and also when the base $\cH_0$ is changed.
\end{enumerate}
\end{definition}
The integration along the fiber alluded to in (3) involves first restriction from $\cH$ to $\cK$, and then
usual integration along the fiber against a (perhaps complex valued) density.
The change of base property will be applied to the submersions $t\circ \pi_1:\cH^{(i)}\to \cH_0$.

\begin{example}\label{ex:weyl2} (Weyl's unitary trick II)
Let $\cH^r\subset \GG\ltimes \EE$ be the local holomorphic groupoid described in Example \ref{ex:weyl1}. Let $\mu_G$
denote the Haar density on $G$ and let
$\mu_{\GG}:=\mu_G\otimes 1\in \Gamma(\cD\otimes \underline{\C})$, where $\cD$ denotes the density bundle of $G$.
We use the product structure of the core $\cK=G\times \EE^r$ to promote  $\mu_{\GG}$ into
a section of $\Gamma(\cD_s\otimes \underline{\C})$, where $\cD_s$ denotes the bundle of densities along the (source) 
fibres of $\cK$. The result is a holomorphic Haar density on $\cK$. 
\end{example}

\section{The s-proper core and holomorphic Haar measure for complexifications of linear models}\label{sec:models}
We aim at using s-proper cores supporting Haar densities for averaging purposes. Therefore, the larger they are, the
more information they will capture of the local Lie groupoid. In other words, we would
like to use the analogues of compact maximal subgroups of a Lie group. In the holomorphic setting
the compatibility of integration with holomorphic structures is rather delicate. We will
content ourselves
with describing s-proper cores and holomorphic Haar densities for the complexifications
of the building blocks of s-proper Lie groupoids.

Let $\cO$ be an orbit of an s-proper Lie groupoid. The linear model for the Lie groupoid around $\cO$   
is given by   
\[\cG\tto \cG_0:=(P\times P\times E)/G \tto P\times_G E,\]
where $\pi:P\to \cO$ is a compact principal $G$-bundle, $G\ltimes E$ is an orthogonal representation, and the
action of $G$ in the above product manifolds is the diagonal one \cite{CS}. 

Let us assume that $\pi:P\to \cO$ is a real analytic principal bundle. Then  $\cG\tto \cG_0$
is a real analytic groupoid with a well defined (germ of) complexification. Moreover, if we fix $\pi:\PP\to \OO$
as a representative of the complexification of $P$, 
then the holomorphic groupoid
\[\cH\tto \cH_0:=(\PP\times \PP\times \EE)/\GG \tto \PP\times_\GG \EE\]
 represents the germ. We shall define
 a family of local holomorphic groupoids $\cG\subset \cH^r\subset \cH$ together
with an s-proper core $\cK\subset \cH^r$ supporting
a holomorphic Haar density.

\subsection{Kempf-Ness theory for principal bundles}
Kempf-Ness theory is centred around the differential geometry (Kahler, Stein, symplectic)
of geometric invariant theory (see for example \cite{He}). We need a few basic results to 
construct our s-proper cores.

Let $G$ be a compact Lie group endowed with a bi-invariant metric, let $\g_\C=\g\oplus i\g$ 
be the complexification of its Lie algebra, and let $\GG$ be the complexified Lie group, in which $G$ embeds.
The Cartan decomposition
\[i\g\times G\to \GG,\qquad (iv,g)\mapsto \exp(iv)g\]
endows the tangent bundle $TG$ with a canonical complex structure. 
Let $TG^r$ be the open disk bundle of radius $r>0$. It is closed by inversion and the right action of $G$ by isometries.
Each right $G$-orbit
is regarded as a point in the symmetric space $\GG/G$. Under this correspondence $TG^r$ maps to 
the ball $B_r$ of radius $r$ centred at the identity/origin. 
For any right $G$-orbit we
can speak of its distance to the identity/origin $eG$
and of its distance to the boundary of the ball $B_r$, which we denote by $\varrho$.
If $\zeta\in \GG$, then 
its (right-)translation by the tube of radius $\varrho=\varrho_\zeta$, denoted by $\zeta \cdot TG^{\varrho}$,
is a $G$-saturated open set whose image in $B_r\subset \GG/G$ is a contractible open subset.

Let us now fix an analytic $G$-invariant Riemannian metric on $P$ whose restriction 
to each $G$-fiber is a bi-invariant metric on $G$.
Let $TP^r$ the open disk bundle of radius $r$ of the tangent bundle. The radius is chosen small enough  so that 
$TP^r$ carries the unique complex structure for which the distance to $P$ function solves the complex Monge-Amp\`ere equation
away from the zero section $P$ \cite[Page 562]{GSt}, \cite[Section 2]{LS}. There is a biholomorphism from $TP^r$ to an open subset of $\PP$, so we may write $TP^r\subset \PP$.
Let $\cK\cN\subset TP^r$ be the Kempf-Ness set (submanifold). We refer the reader to \cite[Theorem 1]{Ag} and  
\cite[Section 5.4]{He} for details
and properties of this submanifold which solves a Kahlerian extremal problem that can be phrased using
both moment map theory and fiberwise plurisubharmonic functions.

\subsection{The local holomorphic groupoids and s-proper cores}\label{ssec:representative}
We will now construct the local holomorphic groupoids $\cH^r\tto  \cH_0^r$. From now on $r>0$ is small
enough so that $TP^r$ carries the complex structure discussed above:

\begin{enumerate}[(i)]
 \item  On the submanifold 
\[\cK\cN\times \EE^r\subset TP^r\times \EE^r\subset  \PP\times \EE^r\]
we consider the diagonal action map: $\GG \times(\cK\cN\times \EE^r)\to TP^r\times \EE^r$.
 Because the submanifold
$G\times \cK\cN\times \EE^r$ has image inside of $TP^r\times \EE^r$, by continuity of the action the inverse image of $TP^r\times \EE^r$
 is an open neighbourhood of $G\times\cK\cN\times \EE^r$. Because $TP^r$, $\EE^r$ and $\GG$ are stable by the right $G$-action,
 upon taking the quotient of the neighbourhood by the action of $G\times G$ on $\GG\times TP^r$  we get a neighbourhood
 of the identity section of the trivial homogeneous bundle
 $\GG/G\times \GG/G\times \OO\times \EE^r$.
 This implies that we may define two `radial' functions
 \[\varrho_s:\OO\times \EE^r\to (0,\infty),\quad \varrho:B_{\varrho_s}\times \OO\times \EE^r\to (0,\infty)\]
      such that:
 
 \[\begin{aligned}\cI_s &:=TG^{\varrho_s}\cdot(\cK\cN\ \times \EE^r)\subset TP^r\times \EE^r,\\
 \cI &:=TG^{\varrho}\cdot (\cK\cN\times \cI_s)\subset TP^r\times \cI_s,\\
\cI_t &:=TG^{\varrho|_e}\cdot(\cK\cN\ \times \EE^r)\subset \cI_s,
\end{aligned}
\]
where $\varrho|_e$ is the restriction of $\varrho$ to the identity section of $B_{\varrho_s}\times \OO\times \EE^r$. By construction
the open subsets $\cI,\cI_s,\cI_t$ are $G$-invariant.
\item  On $\cI$ there is a holomorphic foliation whose leaves are the $TG^{\varrho}$-tubes. We denote its leaf space
by $\cH^r$. On  $\PP\times \PP\times \EE$ there is a holomorphic foliation given by the diagonal $\GG$-action. The
 inclusion $\cI\rightarrow \PP\times \PP\times \EE$ is a map of holomorphic foliations,
 takes leaves of $\cI$ to open and connected subsets of leaves of $\PP\times \PP\times \EE$,
and  is injective on leaf spaces. By construction, $\cK\cN\times  \cI_s\subset \cI$ and  $\cK\cN\times \PP\times \EE\subset \PP\times \PP\times \EE$
are full slices to the holomorphic foliations in which the residual foliation is given by the $G$-action. Therefore the (open) injection on leaf spaces is
\[\cH^r\equiv\cK\cN\times_G \cI_s\subset (\cK\cN\times\PP\times \EE)/G\equiv \cH.\]
% If we denote $(\cK\cN\times \cK\cN\times \EE^r)/G$ by $\cH^r_{\cK\cN}$, then it follows that $\cH^r_{\cK\cN}\subset \cH^r $.
\item On $\cI_s$ and $\cI_t$ there are holomorphic foliations whose leaves are the $TG^{\varrho_s}$ and $TG^{\varrho|_e}$-tubes, respectively.
Let $\cH^r_0$ denote the leaf space of the first one. On $\PP\times \EE$ there is a holomorphic foliation given by the diagonal $\GG$-action.
The inclusion $\cI_s\rightarrow \PP\times \EE$ is a map of holomorphic foliations.
Because leaves go to open connected subsets of leaves, there is an induced map on leaf spaces. By construction 
$\cK\cN\times \EE^r$ and $\cK\cN\times \EE$ are full slices with residual foliation given by the $G$ action. Therefore the (open)
injection on leaf spaces
is 
\[\cH^r_0\equiv\cK\cN\times_G \EE^r\subset \cK\cN\times_G \EE\equiv \cH_0.\]
The second projection map $\cI\to \cI_s$ is a surjective submersion of foliations. It 
 descends
 to a surjective submersion of leaf spaces $s:\cH^r\to \cH_0^r$.
The leaf space of $\cI_t$ can also be canonically  identified with $\cK\cN\times_G \EE^r\equiv \cH^r_0$. 
Because $\varrho(\kappa,\zeta,v)\leq \varrho (\kappa,\kappa',v)$, for all $\zeta$ in the $TG^{\varrho_s}$-tube
 of $(\kappa',v)$,
the restriction of first and third projections from $\PP\times \PP\times \EE^r$  to $\cI$ has image 
 $\cI_t$. Thus,  there is an induced surjective submersion
 $t:\cH^r\to \cH^r_0$.

\end{enumerate}

Because $\cH^r\subset \cH$, it follows that if we restrict the inversion and multiplication map to the subsets
of $\cH^r$ and  ${\cH^{r}}^{(2)}$ where the maps have range in $\cH^r$, then  we get a local holomorphic groupoid
structure on $\cH^r\tto \cH_0^r$. The local 
local holomorphic groupoids $\cH^r, r>0$, define a system of neighbourhoods of $\cG_\cO=P\times_G P$ with compact closure.   

% The submanifold $\cH^r_{\cK\cN}\subset \cH^r$ is a  subgroupoid of $\cH^r$. The source
% fiber at $[\kappa,v]\in \cH^r_0\equiv \cK\cN\times_G \EE^r$ can be identified with $\cK\cN$. 
% 
%  Let ${\cH^r}^{(2)}_{\cK\cN}$ be the subset of composable arrows where the first arrow of a pair  is in $\cH^r_{\cK\cN}$.
%  The multiplication map satisfies $m({\cH^r}^{(2)}_{\cK\cN})\subset \cH^r$:
%  if $[\zeta,\zeta',v]\in \cH^r$, then we can take a representative of the form $(\kappa,\xi,w)\in \cK\cN\times \cI_s$.
%  The arrow in $(\cK\cN\times \cK\cN\times \EE^r)/G$ has a representative of the form $(\kappa'',\kappa,\kappa v)$,
%  and the composition is represented by $(\kappa'',\zeta',v)\in \cK\cN\times \cI_s$.
%  
 
We define 
\[\cK:=(P\times \cK\cN\times \EE^r)/G.\]
% =t^{-1}(P\times_G \EE^r)\cap \cH^r_{\cK\cN}\subset \cH^r_{\cK\cN}.\]

\begin{proposition}
 $\cK$ is an s-proper core for $\cH^r\tto \cH^r_0$.
\end{proposition}
\begin{proof}
By construction $\cK$ is a submanifold intersecting every s-fiber transversely. It is real analytic
because $\cK\cN$ is the zero set of the normalised momentum map for the $G$-action on $TP^r$ \cite{He,Ag}, which is a real analytic map.
% Left multiplication by the Lie subgroupoid $(\cK\cN\times \cK\cN\times \EE^r)/G$ leaves $\cH^r$ invariant:
%  if $[\zeta,\zeta',v]\in \cH^r$, then we can take a representative of the form $(\kappa,\xi,w)\in \cK\cN\times \cI_s$.
%  The arrow in $(\cK\cN\times \cK\cN\times \EE^r)/G$ has a representative of the form $(\kappa'',\kappa,\kappa v)$,
%  and the composition is represented by $(\kappa'',\zeta',v)\in \cK\cN\times \cI_s$. Since $\cK= t^{-1}(P\times_G \EE^r)\cap 
%  (\cK\cN\times \cK\cN\times \EE^r)/G$,
%  we have $m({\cH^r}^{(2)}_\cK)\subset \cH^r$.
% 
% $\cK$ ${\cH^r}^{(2)}_{\cK\cN}$ be the subset of composable arrows where the first arrow of a pair  is in $\cH^r_{\cK\cN}$.
% 
% 
% The submanifold $P\times_G \EE^r$ is a full slice to the orbit foliation of  $\cH^r_{\cK\cN}\tto \cH^r_0$. Therefore $\cK$ is transverse to the source
% map of $\cH^r\tto \cH^r_0$.
% By construction the restriction of the target map to $\cH^r_{\cK\cN}$ has constant rank,  with image $N:=P\times_G \EE^r$.
% Note that $\cK_{N}=(P\times P\times \EE^r)/G$ is a subgroupoid. 

Let $[p,\kappa,v]\in \cK$. The $s_\cK$-fiber at $t([p,\kappa,v])=[p,v]$ (the classes with respect to  the diagonal $G$-action)
is $[P,p,v]$. The $s_\cK$-fiber at $s_\cK([p,\kappa,v])=[\kappa,v]$ is $[P,\kappa,v]$. Left multiplication of $[P,p,v]$ by $[p,\kappa,v]$ yields
the $s_\cK$-fiber at $[k,v]$.

An arrow $[\zeta,\zeta',v]\in \cH^r$ can be represented by $(\kappa,\xi,w)\in \cK\cN\times \cI_s$. The composition
$[p,\kappa, w]\cdot [\zeta,\zeta',v]$ is represented by  $(p,\xi,w)\in \cK\cN\times \cI_s$, so $m({\cH^r}^{(2)}_\cK)\subset \cH^r$.
% 
% Because $m({\cH^r}^{(2)}_{\cK\cN})\subset \cH^r$, we also have $m({\cH^r}^{(2)}_{\cK})\subset \cH^r$, where ${\cH^r}^{(2)}_{\cK}$
% is the subset of $\cI^{(2)}$ where  the first arrow of the pair is in $\cK$. 
% 

The source fiber of $\cK$ at $[\kappa,v]$ can be identified with $P$, so $\cK$ is s-proper.
\end{proof}

\subsection{Holomorphic Haar densities on the s-proper core.}
Let $s_\cK:\cK\to \cH_0$ be an s-proper core for a local holomorphic groupoid $\cH\tto \cH_0$.
A holomorphic Haar density on $\cK$ is in particular a section of the bundle of complexified fibred measures. We
start by discussing necessary conditions on the core to obtain such sections as restrictions of s-fibred holomorphic densities 
on the local holomorphic groupoid $\cH$.

More generally, let $q: Q\to B$ be a proper real analytic submersion. There is an integration map from fibred densities on $Q$ to functions on $B$:
\[\int_Q:\Gamma^\omega(\cD_q)\to C^\omega(B).\]

Let $q:X\to Z$ be a holomorphic submersion, not necessarily proper. Let  $Q\subset X$ be a submanifold
transverse to $q$ and intersecting every fiber in a compact submanifold, so that
$q:Q\to Z$ is a proper submersion.
We say that $Q$ defines a totally real analytic (proper) subfibration if for every point in  $Q$ there is
an analytic diffeomorphism of submersions from a neighbourhood of the point in $(X,Q)\to Z$ to $(\C^n\times \R^m,\R^n\times\R^m)\to \R^m$, which is holomorphic on fibres.

\begin{lemma}\label{lem:totallyreal} Let $q:X\to Z$ be a holomorphic submersion  and let $\cL\to X$ be a square root of the canonical bundle
of $X$. If $Q\to  Z$ is a totally real analytic (proper) subfibration of $q:X\to Z$ then:
\begin{enumerate}
 \item There is a canonical isomorphism of real analytic complex line bundles:
 \[\cL|_Q\to \cD_{q}\otimes \underline{\C}\otimes W,\]
 where $W$ is the flat real line bundle defined by a section of the $\Z_2$-valued degree two fiberwise cohomology bundle of $q:Q\to Z$.
 \item If $W$ is trivial, then $Q$ defines an integration map:
 \[\int_Q:\Gamma^\omega(\cL)\to C^\omega(Z,\C).\]
\end{enumerate}
\end{lemma}
\begin{proof}
Let us fix a system of charts of $X\to Z$ around $Q$ of the form  $(\C^n\times \R^m,\R^n\times\R^m)$. The system defines the corresponding cocycles
$d_{\alpha\beta}\otimes 1,l_{\alpha\beta}$ for $\cD_{q}\otimes \underline{\C}$ and $\cL$ around
$Q$, respectively. The cocycle $l_\alpha\beta$ is a square root of the (fiberwise relative to $\R^m$) complex Jacobian of the change of complex coordinates.
Over the points of $\R^n$ this
Jacobian is the Jacobian for the change of coordinates in $\R^n$. The cocycle $d_{\alpha\beta}\otimes 1$ amounts to composing
the real Jacobian with the branch of the complex square root preserving the positive half line. Thus the cocycles over $\R^n$ only differ
in a sign. The sign defines a cocycle in 
the flat bundle of fiberwise degree 2 cohomology classes with coefficients in $\Z_2$, and thus the flat real line bundle $W$.

If $W$ is trivial, then $\int_Q$ is defined by restricting a real analytic section of $\cL$ to $Q$, using the above isomorphism
to regard it as a complex valued fibred density, and  integrating the complex valued fibred density against the fibres of $q$.
 \end{proof}

Lemma \ref{lem:totallyreal} can be applied to the s-proper core $\cK\subset \cH^r$:
Let $\cD_s\to \cG=(P\times P\times E)/G$ be the bundle of s-fibred densities. Its complexification  $\cL\to \cH^r$
is a square root of the canonical line bundle along s-fibres.

\begin{lemma}\label{lem:analytic-integration} The s-proper core $\cK$ defines an integration map:
  \begin{equation}\label{eq:intanalytic}
  \int_\cK:\Gamma^\omega(\cL)\to C^\omega({\cH^r_0},\C).
   \end{equation}
\end{lemma}
\begin{proof}
The restriction $\cL|_P$  is the complexification of $\cD_s$. Thus, if we show that $s_\cK:\cK\to \cH^r_0$ is a totally
real analytic subfibration of $s:\cH^r\to\cH^r_0$, then by Lemma \ref{lem:totallyreal} the result follows.

The restriction of the target map to $\cK$ is a submersion onto $(P\times \EE^r)/G$. At $[p,\kappa,e]\in \cK$ we 
can take a real analytic submanifold
$[p,\kappa,e]\in S\subset \cK$ such that $S$ is a section of $s_\cK$ and
the restriction of $t$ to $S$ is a  submersion.
We select (the image of)
a section $\kappa\in S_1$ of $\pi:\cK\cN\to \OO$ over $A\subset \OO$. We map the projection of $A$ 
by the metric retraction  $\OO\cong T\cO^r \to \cO$ (a real analytic map)
to  a neighbourhood $B\subset \cO$ of $\pi(p)$. We choose a section $p\in S_2$ of $\pi:P\to \cO$ over $B$. This defines
a submersion $\varphi:S_1\to S_2$ and $S=\varphi(S_1)\times S_1\times \EE^r\subset \cK$ satisfies the requirements.

We have two real analytic fibrations with holomorphic fibres: the first is
a neighbourhood of $s_{\cK}^{-1}(s_{\cK}(S))$ in $s^{-1}(s_{\cK}(S))$. The second is  a neighbourhood of $\pi^{-1}(B)\subset P$ in the collection
of fibres of $\PP$ over $B$. Pulling back by the first projection $B\times \EE^r\to B$  and then by the  right action of $S$ 
defines an isomorphism compatible with the holomorphic structures of the fibres from the second fibration to the
first  one.
Because the second fibration is a complexification of the fibres of a real analytic fibration, the same holds for the first one.
\end{proof}

To prove that the integration map in (\ref{eq:intanalytic}) is compatible with holomorphic structures
we need to use some specific properties of our source fibrations and line bundles. These properties
will reduce our argument to
standard averaging of holomorphic functions by a compact group acting by holomorphic transformations.

Let $q:Q\to B$ be a proper submersion. There is a push forward map from densities on $Q$ to densities on $B$:
\[q!:\Gamma^\omega(\cD_Q)\to \Gamma^\omega(\cD_B).\]
Because $\cD_Q$ is canonically isomorphic to $q^*\cD_B\otimes \cD_q$,  each normalised fibred density defines a right inverse to $q!$.
If $Q$ is a principal $G$-bundle, then the fiberwise Haar density $\mu_G$ canonically defines one such right inverse. If $\chi\in \cD_Q$,
then $q!(\chi)=\int_G f d\mu_G\cdot \rho$, where $\rho$ is any Haar density on $\cD_B$, 
 and $\chi=f\cdot q^*\rho\otimes \mu_G$.

More generally, let $H$ be a Lie group with a bi-invariant density $\mu_H$. Let $Q$ and $R$ be analytic
manifolds on
which $H$ acts freely. The $H$-orbits of the action on $Q$ define a locally
trivial fibration: $\varpi:Q\times_H R\to Q/H $.
Any local slice to the action of $H$ on $R$ trivialises $\varpi$. The trivialisation furnished by different
slices differ by fiberwise conjugation. Therefore the bi-invariant density $\mu_H$ induces a canonical trivialisation of the 
bundle of fibred volume forms/fiberwise
canonical bundle,  which we denote
by $\mu_H\in \Gamma^\omega(K_H)/\mathrm{Hol}(K_H)$.

We apply the previous considerations to both  $\cG=P\times_G (P\times E)$ and $\cH=\PP\times_\GG (\PP\times \EE)$ 
together with the Haar measure $\mu_G$ on $G$ and its complexification $\mu_\GG$ on $\GG$. The 
result are the fibrations 
\[\varpi\circ \pi_1:\cG\to \cO\times \cG_0,\qquad \varpi\circ \pi_1:\cH\to \OO\times \cH_0\]
and the fibred Haar densities: $\mu_G\in \Gamma^\omega(K_G), \mu_\GG\in \mathrm{Hol}(K_\GG)$. The pair $(K_\GG,\mu_\GG)$ is the 
complexification of $(K_G,\mu_G)$.  The bundle  $\cD_s\to \cG$ is the pullback of the
bundle of fibred densities of the trivial bundle $\cO\times \cG_0\to \cG_0$ tensored with $K_G$. Upon complexifying,
we get the isomorphism of holomorphic line bundles over $\cH^r$
\[\cL\cong \overline{\varpi}^*\cD_\OO\otimes K_\GG,\]
where $\overline{\varpi}:\cH^r\to \OO$ is the composition of $\varpi$ with the first projection onto $\OO$.

Let ${\cH^r}^{(d)}_{\cK}\subset {\cH^r}^{(d)}$ denote the tuples of $d$ composable arrows 
in $\cH^r$ where the first $d-1$ arrows  of a tuple belong to $\cK$.

\begin{proposition}\label{pro:holomorphic}
\quad
 \begin{enumerate}
 \item
The integration map defined by the s-proper core  $\cK$ sends holomorphic densities to holomorphic functions:
  \[\int_\cK:\mathrm{Hol}(\cL)\to \mathrm{Hol}(\cH^r_0).\]
  
   \item If $\mu$ is a real analytic Haar system on $\cG$, there exists $r_\mu>0$
   such that the complexification $\mu_\C$ defines a holomorphic
   Haar system on $\cK\subset \cH^r$, for $0<r<r_\mu$:
  \[\int_{\cK}{\mu_\C} :\mathrm{Hol}(\cH^r)\to \mathrm{Hol}(\cH^r_0).\]
  Under the identification $\cL|_{\cK}\cong \cD_{s_\cK}\otimes \underline{\C}$,  $\mu_\C$ maps to a section of $\cD_{s_\cK}$.
  \item
 If $f\in \mathrm{Hol}(\cH^r,\C^m)$,
 then \[\int_{\cK}(m^*f)\mu_\C\in \mathrm{Hol}({\cH^r}^{(d-1)}_{\cK},\C^m).\]
 \end{enumerate}
  \end{proposition}

\begin{proof}
Let us fix $\rho$ a density on $\cO$ and let $\rho_\C$ denote its complexification.
In order to have $\rho_\C$ defined everywhere we may need to shrink $\OO$ and for that it is enough to start with a small enough
radius $r>0$.  We obtain 
a trivialisation $\overline{\varpi}^*\rho_\C\otimes \mu_\GG\in \mathrm{Hol}(\cL)$.
If $\chi\in \mathrm{Hol}(\cL)$, then we write 
\[\chi= f\cdot \overline{\varpi}^*\rho_\C\otimes \mu_\GG,\qquad f\in \mathrm{Hol}(\cH^r).\]
If $z\in \cH^r_0$, then by Lemma \ref{lem:analytic-integration}  $\int_{\cK_z}\chi$ is the result of restricting $\chi$ to $\cK_z$, using the canonical
identification $\cL|_z\cong \cD_{\cK_z}\otimes \underline{\C}$, and integrating the resulting complex valued
density. Because $\overline{\varpi}:\cK_z\to \cO$ is a fibration and because $\mu_\GG$ is bi-invariant,
the canonical identification $\cL|_{\cK_z}\equiv \cD_{\cK_z}\otimes \underline{\C}$
maps $\overline{\varpi}^*\rho_\C\otimes \mu_\GG$ to
$\overline{\varpi}^*\rho\otimes \mu_G$. Therefore:
\[\int_{\cK_z}\chi=\int_{\cO}\left(\int_G f|_{\cK_z}\mu_G\right)\rho.\]
Let $F$ be the pullback of $f$ to $\cI$. This is a holomorphic function
which is constant on the leaves/tubes of the foliation:
\[ \cI=TG^{\varrho}\cdot (\cK\cN\times \cI_s)\subset TP^r\times \cI_s.\]
The manifold $\cI\subset TP^r\times TP^r\times \EE^r$ is invariant by the holomorphic action of $G$ on the first factor.
The $G$-average 
of $F$ is a holomorphic function on $\cI$. Because $\mu_G$ is bi-invariant
this function is constant on the core $G$ of the tube  $TG^{\varrho}\cdot(\kappa,\zeta,v)$ over $z$. Therefore 
 it is constant on the $TG^{\varrho}$-tube and descends to a holomorphic function on $\OO\times \cH^r_0$. This
function is $\int_G f|_{\cK}\mu_G$, and therefore 
\[\int_{\cK}\chi=\int_{\cO}\left(\int_G f|_{\cK}\mu_G\right)\rho\]
is a holomorphic function on $\cH^r_0$, and this proves (1). 

If $\mu$ is a real analytic Haar system on $\cG$, then for $0<r<r_\mu$ the complexification
of $\mu$ is defined everywhere. By (1) $\int_\cK \mu_\C$ is a holomorphic function whose value on $\cG$ is 1. Thus
it is the constant function 1 on $\cH^r$. 

Because $\mu$ is right invariant, $\mu_\C$ is also right invariant. As the trivialisation 
of $\cL|_\cK\cong \cD_{s_{\cK}}\otimes \underline{\C}$ uses the right action of fibres in $\cK$, it follows that
$\mu_\C$ is  a (real valued) density, i.e., a section of $\cD_{s_\cK}\otimes 1\subset \cD_{s_\cK}\otimes \underline{\C}$.

If $\phi:Z\to \cH_0$ is a base change, then  $\phi^*\cL\cong \phi^*\overline{\varpi}^*\cD_\OO\otimes K_\GG$
is trivialised by $\phi^*\varpi^*\rho_\C\otimes \mu_\GG$, 
and therefore the same arguments apply.

A small variation of the change of base applies to the multiplication on the left by $\cK$ several times.
Because $m({\cH^r}^{(d)}_{\cK})\subset \cH^r$ and the multiplication map is continuous, there exists
a neighbourhood of ${\cH^r}^{(d)}_{\cK}$ in ${\cH^r}^{(d)}$ sent into $\cH^r$ by the multiplication map.
This neighbourhood is not exactly the base change of $\cH^r\to \cH^r_0$ by the target map
of ${\cH^r}^{(d)}$. Such smaller neighbourhood can be constructed using  a `radial'
function which depends on the radial functions and distance to the boundary 
of the points of 
the neighbourhood for ${\cH^r}^{(d-1)}$. For such smaller neighbourhood the above arguments apply.

For the proof for $\C^m$-valued holomorphic maps one replaces $\underline{\C}$ by $\underline{\C}^m$.
\end{proof}

\begin{remark} (Local versus global associativity)
In a local Lie group(oid) local associativity refers to associativity for the multiplication of three elements. Global
associativity refers to the multiplication of any finite number $d$ of elements (all $d$-fold products are equal). In general local associativity does
not imply global associativity \cite[Sections 1 and 2]{Go}. In our setting as $\cH^r$ is a subset of a holomorphic groupoid local associativity
implies global associativity. The neighbourhoods of 
${\cH^r}^{(d)}_{\cK}$ in ${\cH^r}^{(d)}$ we consider are built by multiplying a $d$-tuple from right to left.
\end{remark}

\section{Average against s-proper cores and convergence of the iterations}\label{sec:average}
In this section  Zung's construction to correct almost morphisms of s-proper Lie groupoids to compact groups  \cite[Section 2.4]{Zu} 
is generalised as follows:  Almost morphisms  from local analytic groupoids with s-proper cores supporting an analytic Haar density
to (local) Lie groups, are corrected to gain the morphism property with respect to the left multiplication by arrows in the core. 
This implies the compatibility of Zung's construction with $C^\w$-structures
defined on saturated open subsets.

Let $H$ be a Lie group and let $|\cdot |$ be a norm on $\h$ normalised so that on the closed unit ball $B_1(0)$ the exponential map is a diffeomorphism
and $|[u,v]|\leq |u||v|$. 
There are constants depending on the BCH formula such that for vectors $u,v\in B_1(0)$:
\[\begin{aligned}& |\log(\exp(u)\exp(v))-(u+v)|\leq c|u||v|, & c=c(\h,b),\\
   & |\log(\exp(u)\exp(v))|\leq c'|u+v|, & c'=c'(\h,b),\\
& |\exp(u)\exp(v)\exp(-u)|\leq c''(|v|+|v||u|), & c''=c''(\h,b).
\end{aligned}\] 
On $H$ we define a Finsler metric by \emph{left} translation of the norm in $\h$.
The defect of $\exp$ and $\log$ from being isometries between $B_1(0)$ and $B_1(e)$ is measured by Lipschitz 
constants:
\[d\leq |\exp|\leq d', \qquad d=d(b,1),\,d'=d'(b,1).\] 
Let $W\subset H$ be an open neighbourhood of the identity containing $B_1(e)$, and let $K$ be a compact subset whose interior contains the closure of $W$.
For each $h\in K$ consider the adjoint action of $h$ on $H$. The maximum of the norm of the derivative over $B_1(e)$ 
measures how far the map is from being an isometry. We define $c_l:=c_l(K,b)$ to be the maximum of such constants over $h\in K$.
We define $c_d:=c_d(b)$ to be the distance from $\overline{W}$ to the boundary of $K$.
We have:
\[K\cdot B_{1/c_l}(e)\cdot K^{-1}\subset B_1(e),\qquad B_{1/c_d}(e)\cdot W\subset K.\]
 
Let $\cH\tto M$ be a local  groupoid. Let $\cK\subset \cH$ be an s-proper core for $\cH$ together with an analytic Haar density $\mu$. If
$\phi:\cH\to H$ is an analytic map, then following Zung (but using left instead of right multiplication)  we define:
\[\psi:\cH^{(2)}_\cK\to H,\qquad \psi(k,p)=\phi(p)^{-1}\phi(k)^{-1}\phi(kp),\qquad \Delta(\phi)=d(\psi(\cH^{(2)}_\cK),e).\]
If $\Delta(\phi)\leq 1/c_l$, then we define the averaging against $\cK$ and the correction by the averaging:
\[A_\cK(\phi)=\exp\left(\int_{\cK}\log(\psi) \mu\right),\qquad \hat{\phi}=\phi\cdot A_\cK(\phi).\]
 \renewcommand\thetheo{\Alph{theo}}% locally redefine the representation of the theorem counter

\begin{theo}\label{thm:average} Let $H$ be a Lie group endowed with a norm $h$ on its lie algebra as above, 
and let $W\subset K$ be subsets of $H$ as above.  
 
 There exists a constant $C:=C(H,h,W,K)>0$,
such that 
\begin{itemize}
 \item for any local analytic groupoid $\cH$ 
endowed with an s-proper core $\cK\subset \cH$,
\item  for any
Haar density $\mu$ on $\cK$, and
\item  for any analytic function $\phi:\cH\to H$ such
that $\phi(\cH)\subset W$ and $\Delta(\phi)\leq C$,
\end{itemize}
the sequence of iterates of the operator which corrects $\phi$ by 
$A(\phi)$ converges in the $C^0$-norm over compacts to a map:
\[\Phi:\cH\to H,\qquad 
\Phi(kp)=\Phi(k)\Phi(p),\quad \forall (k,p)\in \cH^{(2)}_\cK.\]

\end{theo}

\begin{proof}
We have
$|A_\cK(\phi)|\leq \frac{d}{d'}\cdot C$. Let $\hat{\psi}=\hat{\phi}A_\cK(\hat{\phi})$.
The proof of Zung --with two minor variations-- applies verbatim to show the existence of a polynomial $q$  with
$|\hat{\psi}|\leq q(C)$ (and $q(0)=0$). The first variation 
is the use of the right invertibility of $s_\cK$-fibres so that \cite[Equation (2.29)]{Zu} still holds. The second
variation is the introduction of an extra  multiplicative constant in \cite[Equation (2.32)]{Zu}
which accounts for the failure of the adjoint action of $K$ on $B_1(e)$ to be an isometry.

More precisely, we can choose:
\[ q(C)=\frac{2\cdot c \cdot c_l (d'\cdot c_l+2d'+c\cdot c_l\cdot C)}{{d'}^2}\cdot C^2.\]
Thus, there exists a choice of $C>0$ such that the iterations $\phi_{n+1}:=\hat{\phi_n}$ satisfy
\[|A_\cK(\phi_n)|\leq \frac{1}{c_l},\qquad |\phi_{n+1}\phi_n^{-1}|\leq  \frac{1}{c_d}\]
and $\phi_n$ converges in the $C^0$-norm to $\Phi: \cH\to K\subset H$
\[\Phi(kp)=\Phi(k)\Phi(p),\quad \forall (k,p)\in \cH^{(2)}_\cK.\]
\end{proof}
\begin{remark} There is a version of \ref{thm:average} for maps $\phi:\cH\to H$ where exhaustions of $H$ by
compact subsets and of $\cH$ by local analytic subgroupoids have to be used. 
\end{remark}

As announced, we can now prove the compatibility of Zung's iteration with $C^\w$ structures:
\begin{corollary}\label{cor:zunganal} Let $G$ be a compact Lie group endowed with an $\Ad$-invariant inner product $g$
on its lie algebra. There exists a constant $C:=C(K,g)>0$
such that
\begin{itemize}
 \item for any Lie groupoid $\cG$ together with and open saturated subset $\cG_U\subset \cG$
endowed with a real analytic structure,
\item  for any
Haar density $\mu$ on $\cG$ which is real analytic on $\cG_U$, and
\item for any function $\phi:\cG\to K$ 
real analytic on $\cG_U$ and such that $\Delta(\phi)\leq C$,
\end{itemize} 
the sequence of iterates of the operator which corrects $\phi$ it by 
$A_\cG(\phi)$ converges in the $C^\infty$-norm over compacts to a morphism of real analytic groupoids:
\[\Phi:\cG\to G,\qquad 
\Phi(qp)=\Phi(q)\Phi(p),\quad \forall (p,q)\in \cG^{(2)}.\]
\end{corollary}

\begin{proof}
Let  $TG^1, TG^2\subset \GG$ be the tubes of radius one and two, say. We define
\[C:=\frac{1}{2}\mathrm{min}\{C(\GG,g\oplus g,\overline{TG^2},TG^1),C_Z(K,g)\},\]
where the first constant comes from Theorem \ref{thm:average} 
and $C_Z(K,g)$ is the constant in Zung's theorem needed for converge on the $C^\infty$ norm \cite[Section 2.5]{Zu}.

We fix $\cH_U$ a small enough representative of the germ of the complexification of $\cG_U$ so that the the following properties hold:
\begin{enumerate}[(i)]
 \item the complexification of $\phi|_{G_U}:\cG_U\to G$, which we denote
by $\phi_\C$, 
is everywhere defined and has image inside $TG^1$;
 \item $|\Delta(\phi|_\C)|\leq C$.
 \item the complexification of $\mu$ is everywhere defined: $\mu_\C\in\mathrm{Hol}(\cL)$.
\end{enumerate}
Because $C\leq C_Z(K,g)$ the sequence $\phi_{n}$ (obtained using averaging against the whole $\cG$) 
converges in the $C^\infty$-norm to 
the Lie groupoid morphism $\Phi:\cG\to G$. To show
that $\Phi|_{\cG_U}$ is real analytic we choose any point $y\in U$. We let $\cG_V$
denote the restriction of $\cG$ to $V$, which is a small enough saturated neighbourhood of 
the leaf through $y$. The normal form Theorem for s-proper Lie groupoids \cite{CS} implies that:
\begin{equation}\label{eq:proplin}
 \cG_V\tto V\cong (P\times P\times E^r)/G\tto P\times_G E^r.
\end{equation}
It is easy to check that the isomorphism (ref{eq:proplin}) built in \cite{CS} is compatible with real analytic data which is
\emph{everywhere} defined on $\cG_V$. 

We let $\cH^r_V$ be the representative of the germ of complexification 
of $\cG_V$ constructed in subsection \ref{ssec:representative}. The radius $r>0$ is chosen so that $\cH^r_V$ is biholomorphic to a subset of $\cH$. Thus
the triple $(\cK\subset \cH^r_V,\phi_\C|_{\cH^r_V},\mu_\C|_{\cH^r_V})$ is in the hypotheses
of Theorem \ref{thm:average} in the holomorphic category. Therefore
the sequence of holomorphic maps which is obtained out of $\phi_\C|_{\cH^r_V}$
converges in the $C^0$-norm over compacts to a holomorphic map
\[\Phi^r_\C: \cH^r_V\to \GG.\]
The restriction of the core $\cK\subset \cH^r_V$ and the holomorphic Haar measure to $\cG_V$ is 
by construction $(\cG_V,\mu|_{\cG_V})$. Hence $\Phi^r_\C|_{\cG_V}=\Phi|_{\cG_V}$. This shows 
that $\Phi$ is real analytic on $\cG_U$.

Observe that a fortiori  $\Phi^r_\C$ is a morphism of holomorphic groupoids (averaging against the 
core $\cK$ in principle only grants the morphism property when left multiplying by
arrow in $\cK$).

\end{proof}

\section{Proof of Theorem \ref{thm:main}}\label{sec:proof}

The (saturated) semilocal normal forms for proper Lie groupoids are built in two steps \cite{CS}: the first one
is the linearization of an s-proper Lie groupoid around a fixed point. The second step  
amounts to the use of appropriate Morita equivalences. We only need to clarify the compatibility
of the latter step with $C^w$ structures defined on saturated open subsets.

The following lemma is probably well-known to experts:
\begin{lemma}\label{lem:palais} Let $G$ be a compact Lie group acting on submersion $\pi:P\to E$, where  $E$ is a vector
space and the action of $G$ on $E$ is orthogonal.
If $G\ltimes (\pi:P\to E)$ admits a compatible $C^\w$ structure on an open subset $W$ saturated both by $s$-fibres and $G$-orbits,
then the compatible $C^\w$ structure extends to $G\ltimes (\pi:P\to E)$.
\end{lemma}

\begin{proof}
Because $G$ is compact by the Mostow-Palais Theorem  there exists an orthogonal representation
$G\ltimes E'$ and a close equivariant embedding  $i:P\to E'$. Because $C^\w$ equivariant
morphisms are dense in the very strong
topology \cite[Theorem 15.5]{Il2}
we may assume that the embedding is analytic on $W$.

The product embedding $\iota:=\pi\times i:P\to E\times E'$ is equivariant with respect to  the diagonal
action on the product of vector spaces.  Because $\iota$
is transverse to the first projection $\pi_1:E\times E'\to E$, we may take  a small tubular neighbourhood $\nu$ 
of $\iota(P)$ with  respect to the fiberwise Euclidean metric in $E\times E'\to E$.

Let $c$  be the codimension of $\iota(P)$
and let $\U$ be the universal bundle of the Grassmann (real analytic) manifold of planes of dimension $c$ in $E\times E'$.
The fibred Pontryagin map 
$l:\nu\to \U$
is  equivariant,
its image lies in the subbundle $\U_{E}$ of planes orthogonal to $E$, and it is transverse to its zero section. 
We take  $\Theta\in C^\omega(\nu, \U_E)^G$ close enough to $l$, so that it is still transverse to the zero section. The analytic
submanifold $\Theta^{-1}(\underline{0})$ is  transverse to $\pi_1$. The fibred orthogonal retraction
 $\lambda:\Theta^{-1}(\underline{0})\to i(P)$ is an equivariant diffeomorphism. Its inverse is analytic on $\iota(W)$.
 
The composition $\lambda\circ \iota^{-1}$ defines is a compatible analytic structure  $G\ltimes(\pi:P\to E)$ which extends
the given one on $W$.
\end{proof}

We are ready to prove our main Theorem:
\begin{proof}[Proof of Theorem \ref{thm:main}]

By Zorn's lemma it is enough to show that a compatible analytic structure on a saturated open subset $\cG_V$ different from $\cG$
can be extended to a larger open saturated subset. If $V$ is a connected component of $M$, then the construction of the extension
is straightforward. Let $x\in \bar{V}\backslash V$. We take $\Sigma\to M$ a small slice to the orbit foliation through $x$. We may assume without loss of generality that
$\Sigma\cap V$ is analytic: a slice $\Sigma$ can be constructed as the fiber of a local submersion $q$ which is
transverse to the characteristic distribution. We may replace the restriction of $q$ to $V$ by an analytic map close enough
in the very strong topology. 

The Lie groupoid $\cG_\Sigma\to \Sigma$ is s-proper  and has a compatible analytic structure on $\cG_{\Sigma\cap V}$.
We endow it with a Haar density $\mu$ which can be assumed to be analytic on $\cG_{\Sigma\cap V}$. Let $G$ be the isotropy group
of $x$ and let $g$ be an inner product on $g$. 
Upon reducing $\Sigma$, we may assume that we have a submersion $\phi:\cG_\Sigma\to G$ so that
$\Delta(\phi)\leq C/2$, where $C=C(G,g)$ in corollary \ref{cor:zunganal}. 
We may replace $\phi$ over $\cG_{\Sigma\cap V}$ by a $C^\w$ map close enough in the strong topology,
so that we can assume that  $\phi$  
is a submersion which is $C^\w$ on $\cG_{\Sigma\cap V}$ and satisfies
$\Delta(\phi)\leq C$. 

Because we are in the hypotheses of Corollary \ref{cor:zunganal} we obtain 
\[\Phi:\cG\to G\]
a morphism of Lie groupoids which 
 is $C^\w$ on $\cG_{\Sigma\cap V}$. If
 $P:=s^{-1}(\Sigma)$, then we have a group $G$ of bisections of $\cG_\Sigma$ which acts on $s:P\to \Sigma$ in a compatible way with the
 analytic structure on $s^{-1}(\Sigma\cap V)$.
Upon reducing  $\Sigma$ we may assume that it is a vector space with a linear action as follows: we embed
the $G$-space $\Sigma$ in an orthogonal representation space $i:\Sigma \to E$ in a compatible fashion with the analytic structure on $\Sigma \cap V$.
The tangent space to $i(y)$ is an invariant $G$-module, and therefore the orthogonal projection from $i(\Sigma)$ to the tangent
space is a map of $G$-spaces, which is analytic on $i(\Sigma\cap V)$.

By Lemma \ref{lem:palais} we can extend the compatible analytic structure 
from $G\ltimes (s^{-1}(\Sigma\cap V)\to \Sigma\cap V)$ to $G\ltimes (s:P\to \Sigma)$.
If we denote by $Z$ the saturation of $\Sigma$ by orbits, then we obtain a compatible analytic structure on 
\[\cG_{Z}\cong  P\times_{G\ltimes \Sigma} P\]
which matches the compatible analytic structure on $\cG_{Z\cap V}$.

\end{proof}

\begin{remark} (Lower regularity)
The following result holds: 

For any Lie proper Lie groupoid  $\cG\tto M$ of class $C^k$, $k\geq 3$,
there exists a $C^\w$  structure compatible with the $C^{k-1}$-structure.

To improve the regularity from $C^k$ to $C^\infty$ we apply Zung's linearization Theorem, which 
is  valid in the $C^k$ class, $0\leq k<\infty $. A degree of regularity is lost because as initial data one needs a Haar density
of class $C^k$. Such densities come from sections of the Lie algebroid, which is of class $C^{k-1}$. 

The requirement $k\geq 3$ is explained as follows: Zung's linearization Theorem applies to s-proper groupoids with fixed points.
To show that the groupoid $G_\Sigma$ coming from a slice is s-proper one actually needs to resort to either of the proofs
of the linearization
Theorem which use an infinitesimal approach \cite{CS,HF}. There, an additional degree of regularity is lost
(for example in \cite{CS} this is a consequence of applying the Alexander trick).
 \end{remark}

 \end{document}